\documentclass[11pt,a4paper,twoside]{amsart}
\usepackage{amsmath,amssymb,amsfonts,amsthm}
\usepackage{hyperref}

\def\to{\longrightarrow}

\def\n{\nabla}

\def\L{\bigtriangleup}

\def\<{\langle}
\def\>{\rangle}

\def \D{{\mathcal{D}}}

\def \H{{\mathcal{H}}}

\newtheorem{The}{Theorem}[section]
\newtheorem{Pro}[The]{Proposition}
\newtheorem{Lem}[The]{Lemma}

\theoremstyle{definition}

\newtheorem{Rem}[The]{Remark}


\begin{document}
\title{Sobolev $H^1$ Geometry of the Symplectomorphism Group}
\author{J. Benn and A. Suri }
\address{School of Engineering and Advanced Technology, Massey University,\\
Palmerston North, New Zealand.}
\email{j.benn1@massey.ac.nz}
\address{Department  of Mathematics, Faculty of sciences \\
Bu-Ali Sina University, Hamedan 65178, Iran.}
\email{ali.suri@gmail.com \& a.suri@basu.ac.ir}
\maketitle {\hspace{2.5cm}}

\begin{abstract}
For a closed  symplectic manifold $(M,\omega)$ with compatible Riemannian metric $g$ we study the Sobolev $H^1$ geometry of the group of all $H^s$ diffeomorphisms on $M$ which preserve the symplectic structure. We show that, for sufficiently large $s$, the $H^1$ metric admits globally defined geodesics and the corresponding exponential map is a non-linear Fredholm map of index zero. Finally, we show that the $H^1$ metric carries conjugate points via some simple examples.

\end{abstract}

\pagestyle{headings} \markright{Symplectomorphisms  with $H^1$ metric}
\tableofcontents

%
%
%
%
\section{Introduction}
A number of well-known nonlinear partial differential equations of mathematical physics arise as geodesic equations on various infinite dimensional Lie groups. The most fundamental example is the Euler equations of ideal hydrodynamics where Arnold \cite{Arnold66} noticed that a curve $\eta(t)$ in the group of smooth volume preserving diffeomorphisms is a geodesic of the right-invariant $L^2$ metric if and only if the velocity field $v(t)=\dot{\eta(t)}\circ \eta^{-1}(t)$ solves the Euler equations of hydrodynamics. Subsequently, Ebin and Marsden \cite{Ebin-Marsden} showed that the volume preserving diffeomorphism group can be given the structure of a Hilbert manifold and proved that the corresponding geodesic equation is actually a smooth ordinary differential equation, hence rigourously justifying the infinite dimensional geometric constructions. Since then, much research has focused on finding geometric formulations for other conservative systems in continuum mechanics. For examples of these see Arnold-Khesin \cite{Arnold98}, Marsden and Ratiu \cite{MarsdenRatiu}, Schmid \cite{Schmid}, Taylor \cite{Taylor-Ev}, Vizman \cite{Vizman}, Zeitlin and Kambe \cite{ZK}, or Khesin-Wendt \cite{KhesinWendt} and the references therein.

Two motivations for the present work are the Lagrangian-averaged Euler equations as considered by Shkoller \cite{Shkoller98} and the symplectic diffeomorphism group equipped with the $L^2$ metric as considered by Ebin \cite{Ebin}.

The Lagrangian-averaged Euler equations were first introduced by Holm, Marsden, and Ratiu in Euclidean space, and describe the mean hydrodynamic motion of an incompressible fluid. Shkoller \cite{Shkoller98} realized the Lagrangian-averaged Euler equations as the Euler-Poincare equations associated with the geodesic flow of the right-invariant Sobolev $H^1$ metric on the group of volume-preserving diffeomorphisms of a compact n-dimensional Riemannian manifold. He showed that, for sufficiently large $s$, the geodesic spray is continuously differentiable so that a standard Picard iteration argument proves existence and uniqueness of solutions on a finite time interval. Kouranbaeva and Oliver \cite{KO} have subsequently shown that when the underlying manifold is two dimensional the Lagrangian-averaged Euler equations admit weak global solutions.

In the present work we consider a closed symplectic   manifold, with symplectic form $\omega$, and the group of Sobolev $H^s$ symplectic diffeomorphisms $\mathcal{D}^{s}_{\omega}(M)$ of $M$ equipped with the right-invariant Sobolev $H^1$ metric on vector fields. Our main theorem is
\begin{The}
Let $(M,\omega)$ be a closed, $2n$-dimensional symplectic manifold with compatible Riemannian metric $g$. Then, for $s>n+3$, the group $\mathcal{D}^{s}_{\omega}(M)$ is $H^1$ geodesically complete.
\end{The}
We note that when $M$ is two dimensional the symplectomorphism group and volume-preserving diffeomorphism group coincide so that our results strengthen those of Kouranbaeva and Oliver aswell as generalize them to the $2n$-dimensional symplectic setting.

The symplectomorphism group equipped with the $L^2$ metric on vector fields has already been conidered by Ebin \cite{Ebin} who showed that the corresponding geodesic equation is globally well-posed. In this setting the subgroup of Hamiltonian diffeomorphisms play a role in plasma dynamics analogous to the role played by the volume-preserving diffeomorphism group in hydrodynamics. In particular, the $L^2$ geodesic equation is equivalent to the geodesic Vlasov equation (see Holm and Tronci \cite{HolmTronci} and the references therein) and we could consider the geodesic equation of the $H^1$ metric as the Lagrangian-averaged geodesic Vlasov equation.

Ebin \cite{Ebin} also asked if the $L^2$ exponential map is a non-linear Fredholm map of index zero which was answered in the affirmative in \cite{Benn4}. Here we use the general framework of Misiolek and Preston \cite{MP} for studying exponential maps of right-invariant Sobolev $H^r$ metrics and show

\begin{The}
Let $(M,\omega)$ be a closed, $2n$-dimensional symplectic manifold with compatible Riemannian metric $g$. Then, for $s>n+3$, the exponential map of the $H^1$ metric on $\mathcal{D}^{s}_{\omega}$ is a non-linear Fredholm map of index zero.
\end{The}

A bounded linear operator $L$ between Banach spaces is said to be
Fredholm if it has finite dimensional kernel and cokernel.
It then follows from the open mapping theorem that $\text{ran}\, L$ is closed.
$L$ is said to be semi-Fredholm if it has closed range and either its kernel or cokernel
is of finite dimension.
The index of $L$ is defined as
$
\mathrm{ind}\, L = \dim\ker L - \dim\mathrm{coker}\, L.
$
The set of semi-Fredholm operators is an open subset in the space of
all bounded linear operators and the index is a continuous function on this set
into $\mathbb{Z}\cup\left\{ \pm\infty\right\}$, cf. Kato \cite{Kato}.
A $C^1$ map $f$ between Banach manifolds is called Fredholm
if its Fr\'echet derivative $df(p)$ is a Fredholm operator at each point $p$ in the domain of $f$.
If the domain is connected then the index of the derivative is by definition
the index of $f$, cf. Smale \cite{Smale65}.

\hskip 0.5cm
Let $\gamma$ be a geodesic in a Riemannian Hilbert manifold.
A point $q=\gamma(t)$ is said to be conjugate to $p=\gamma(0)$ if the derivative
$D\exp_p(t\dot{\gamma}(0))$
is not an isomorphism considered as a linear operator between the tangent spaces at $p$ and $q$.
It is called monoconjugate if $D\exp_p(t\dot{\gamma}(0))$ fails to be injective
and epiconjugate if $D\exp_p(t\dot{\gamma}(0))$ fails to be surjective.
In general exponential maps of infinite dimensional Riemannian manifolds are not Fredholm.
For example, the antipodal points on the unit sphere in a Hilbert space with the induced metric
are conjugate along any great circle and the differential of the corresponding exponential map
has infinite dimensional kernel.
An ellipsoid constructed by Grossman \cite{Grossman65} provides
another example as it contains a sequence of monoconjugate points along a geodesic arc
converging to a limit point at which the derivative of the exponential map is injective but not surjective.
Such pathological phenomena are ruled out by the Fredholm property
because in this case monoconjugate and epiconjugate points must coincide, have finite multiplicities and cannot cluster along finite geodesic segments. In particular, the $H^1$ exponential map behaves like that of a finite dimensional manifold.

The paper is organized as follows: In the next section we give the functional analytic setting for the symplectomorphism group and the $H^1$ metric on vector fields. We then proceed to show that the $H^1$ exponential map is a local diffeomorphism on a neighbourhood of the identity diffeomorphism; that is, geodesics exist for at least a short time, are unique, and depend smoothly on their initial velocity. In the subsequent section we show that the domain of definition of the $H^1$ exponential map can be extended to the whole tangent space; this proves Theorem 1.1 above.
In the second half of the paper we prove Theorem 1.2. Fredholm properties of Riemannian exponential maps on diffeomorphism groups is now well understood and a general theory has been laid out by Misiolek and Preston \cite{MP}. For this reason we will sketch the argument and indicate where the necessary changes occur. Finally, in section 6 we show that the $H^1$ metric carries conjugate points via some simple examples.
\section{Preliminaries}\label{section preliminaries}

Let $M$ be a closed   symplectic manifold of dimension $2n$, with symplectic form $\omega$ and Riemannian metric $g$. We assume that $\omega$ and $g$ are compatible in the sense that there exists an almost complex structure $J:TM\to TM$ such that $J^2=-1$ and $\omega (v,w)=g(v,Jw)$. The volume form defined by $\omega$ coincides with the volume form induced by $g$ and we denote this volume form by $\mu$. Let $\mathcal{D}^{s}_{\omega}$ denote the group of all diffeomorphisms of Sobolev class $H^s$ which preserve the symplectic form $\omega$ on $M$. When $s>n+1$, $\mathcal{D}^{s}_{\omega}$ becomes an infinite dimensional manifold whose tangent space at the identity $T_e\mathcal{D}^{s}_{\omega}$ is given by $H^s$ vector fields $v$ on $M$ which satisfy $\mathcal{L}_v\omega=0$.

It is convenient to consider $\mathcal{D}^{s}_{\omega}$ as a submanifold of the full diffeomorphism group $\mathcal{D}^{s}$. According to the Hodge decomposition we have
\begin{equation} \label{eq:Hodge}
H^s(T^*M)=d\delta H^{s+2}(T^*M)\oplus\delta dH^{s+2}(T^*M)\oplus \H
\end{equation}
where $\H$ is the space of all $H^s$ harmonic forms on $M$ and the summands are orthogonal with respect to the $L^2$ inner product on vector fields
\begin{equation} \label{eq:L2metric}
\<v,w\>_{L^2}=\int_M g(v,w)d\mu, \qquad v,w \in T_{e}\mathcal{D}^{s}.
\end{equation}
The action of $\mathcal{D}^{s}_{\omega}$ on $\mathcal{D}^s$ by composition on the right is an isometry of the metric \eqref{eq:L2metric} and we obtain an $L^{2}$ orthogonal splitting in every tangent space

$$T_\eta\mathcal{D}^{s} =T_\eta\D^s_\omega \oplus\omega^{\sharp}\delta dH^{s+2}(T^*M)\circ\eta,  $$
where
\begin{equation*}
T_e\D^s_\omega=\omega^\sharp(d\delta H^{s+2}(T^*M)\oplus \H),
\end{equation*}
and $\omega^{\flat}:TM \rightarrow T^{*}M$ is the musical isomorphism defined by $\omega^{\flat}(v) = i_{v}\omega$ with inverse $\omega^{\sharp}:T^{*}M \rightarrow TM$. The projections onto the first and second summands depend smoothly on the basepoint $\eta$ and will be denoted by $P_{\eta}$ and $Q_{\eta}$, respectively, cf. Ebin and Marsden \cite{Ebin-Marsden}.

The positive-definite Hodge Laplacian is defined on vector fields by $\Delta v = \omega^{\sharp}(d\delta\omega^{\flat}v + \delta d\omega^{\flat}v)$. On the full diffeomorphism group $\mathcal{D}^{s}$, define the right-invariant $H^1$ metric at the identity diffeomorphism by
\begin{equation} \label{eq:H1metric}
\langle u,v \rangle_{H^1} = \int_{M} g((I+\Delta)u,v)d\mu, \qquad u, v \in T_e\mathcal{D}^{s}_{\omega}.
\end{equation}
This metric is smooth, as shown in \cite{Shkoller98}, and we obtain, by restriction, a smooth right-invariant Riemannian metric on the group $\mathcal{D}^{s}_{\omega}$.
Since $T_e\mathcal{D}^{s}_{\omega}$ is, up to a finite dimensional subspace, generated by functions, the $L^2$ metric on vector fields corresponds to an $H^1$ metric on Hamiltonian functions and the $H^1$ metric on vector fields corresponds to an $H^2$ metric on Hamiltonian functions. As mentioned in the introduction, it is known that all $L^p$ metrics on Hamiltonian functions are degenerate, but the Sobolev $H^r$ metrics admit smooth exponential maps which are defined on the whole tangent space (see section 4) and whose singularities are similar to those of a finite dimensional manifold (see section 5).

%
The following Lie theoretic operators will be useful in the subsequent sections. Recall that for any $\eta \in \mathcal{D}_{\omega}^{s}$ the group adjoint operator on
$T_{e}\mathcal{D}_{\omega}^{s}$ is given by $\mathrm{Ad}_{\eta} = dR_{\eta^{-1}}dL_{\eta}$
where
$R_{\eta}$ and $L_{\eta}$ denote the right and left translations by $\eta$.
Consequently, given any $v, w \in T_e\mathcal{D}_{\omega}^{s}$ we have
\begin{equation} \label{eq:Ad}
w \to \mathrm{Ad}_{\eta} w
=
\eta_\ast w
=
D\eta \circ \eta^{-1}( w \circ \eta^{-1})
\end{equation}
For $v\in T_e\mathcal{D}^{s}_{\mu}$, the algebra adjoint $\mathrm{ad_{v}}$ on $T_e\mathcal{D}^{s}_{\mu}$ is defined by
\[
\mathrm{ad}_{v}u=J\nabla\omega(v,u).
\]
and is related to the group adjoint by
\begin{equation} \label{eq:gaadjoint}
\frac{d}{dt}\mathrm{Ad}_{\eta^{-1}(t)}u=-\mathrm{\mathrm{Ad}_{\eta^{-1}(t)}}\mathrm{ad}_{v(t)}u,
\end{equation}
where $v(t)=\dot{\eta(t)}\circ \eta^{-1}(t)$.
The associated coadjoint operators are defined using the $H^1$ inner product by
\begin{equation} \label{eq:grpcoAd-1}
\langle \mathrm{Ad}_{\eta}^{*}v,w\rangle _{H^{1}}
=
\langle v,\mathrm{Ad}_{\eta}w \rangle _{H^{1}}
\end{equation}
and
\begin{equation} \label{eq:grpcoad-1}
\langle \mathrm{ad}_{v}^{*}u,w \rangle _{H^{1}}
=
\langle u,\mathrm{ad}_{v}w \rangle _{H^1}
\end{equation}
for any $u, v$ and $w \in T_e\mathcal{D}_{\omega}^{s}$.
The group and algebra coadjoint operators are similarly related via
\begin{equation} \label{eq:gacoadjoint}
\frac{d}{dt}\mathrm{Ad}_{\eta^{-1}(t)}^{*}u=-\mathrm{ad}_{v(t)}^{*}\mathrm{Ad}_{\eta^{-1}(t)}^{*}u.
\end{equation}
Given a Lie group $G$ with a (possibly weak) right-invariant metric $\sigma$ it is well-known that a curve $\eta(t)$ is a geodesic if and only if the curve $v(t)$ in $T_{e}G$, defined by the flow equation
\begin{equation} \label{eq:flow}
\dot{\eta}=v(t)\circ\eta(t),
\end{equation}
satisfies the Euler-Arnold equation
\begin{equation} \label{eq:E-A}
\dot{v}(t)=-\mathrm{ad}^{*}_{v(t)}v(t).
\end{equation}

For a proof of this fact we refer the reader to Arnold-Khesin \cite{Arnold98},  Khesin-Wendt \cite{KhesinWendt} or Taylor \cite{Taylor-Ev}.

%
%
%
%
%
\section{Geodesics of $(\D^s_\omega, \langle ,\rangle_1)$}
We now proceed to form the Euler-Arnold equations on $T_{e}\mathcal{D}^{s}_{\omega}$. In order to prove local existence and uniqueness it is convenient to rewrite the equations in terms of the flow equation \eqref{eq:flow}. We will then show that the resulting second order (geodesic) equation is a smooth ODE on the infinite dimensional manifold $\mathcal{D}^{s}_{\omega}$ and can be solved using a standard Banach-Picard iteration scheme.

The basic technique was developed by Ebin and Marsden \cite{Ebin-Marsden} in their study of the Cauchy problem for the Euler equations of hydrodynamics. The proof we use here is completely analogous to the proofs given in \cite{Ebin-Marsden} and the proofs given by Shkoller \cite{Shkoller98} in his study of the $H^1$ metric on the volume preserving diffeomorphism group. Additionally, the basic theory for geodesics of right-invariant Sobolev metrics on certain diffeomorphism groups has been laid out by Misiolek and Preston \cite{MP}. For this reason we will give an outline of the argument and refer the reader to the aforementioned references.

%

%
%
%
%
\begin{Pro}
The Euler-Arnold equation of the $H^1$ metric \eqref{eq:H1metric} is given by
\begin{equation} \label{eq:ASE}
(1+\L)v_t+  P_{e}(\n_v(1+\L)v+(\n v)^t\L v)=0,
\end{equation}
where $P_{e}$ is the $L^2$ orthogonal projection onto $T_e\mathcal{D}^{s}_{\omega}$.
\end{Pro}
\begin{proof}
Let  $u,v\in T_e\mathcal{D}^{s}_{\omega}$. The Euler-Arnold equation is given by
\begin{equation}\label{eq:avsympeuler}
v_t+\textrm{ad}^*_vv=0
\end{equation}
Using the formula $(\mathcal{L}_X)^*Y=-(\textrm{div}X)Y-(\n X)^tY-\n_XY$ for vector fields $X$ and $Y$, along with the divergence theorem, we have
\begin{eqnarray*}
\<\textrm{ad}^*_vv,u\>_1  &=&  \<v,\textrm{ad}_vu\>_1\\
&=&\int_Mg(\textrm{ad}_vu,(1+\L)v)d\mu\\
&=&\int_Mg(-\mathcal{L}_vu, (1+\L)v)d\mu\\
&=&\int_Mg(u, -\mathcal{L}_v^*(1+\L)v)d\mu\\
&=& \int_Mg  \Big( u, \n_v(1+\L)v +(\n v)^t(1+\L)v+\textrm{div}v(1+\L)v  \Big)d\mu\\
&=& \int_Mg  \Big( u, \n_v(1+\L)v +(\n v)^t(\L)v  \Big)+g  \Big( u, (\n v)^tv  \Big)  d\mu\\
&=& \int_Mg  \Big( u, \n_v(1+\L)v +(\n v)^t(\L)v  \Big)
\end{eqnarray*}
Since the expression in the last line is the $L^2$ metric, and $u \in T_e\mathcal{D}^{s}_{\omega}$, we must compose with the $L^2$ Hodge projection $P_e$. It follows that
\begin{equation*}
\<\textrm{ad}^*_vv,u\>_{H^1} =\langle (1+\Delta)^{-1}P_e (\n_v(1+\L)v +(\n v)^t(\L)v),u\rangle_{H^1}
\end{equation*}
and since this holds for any $u \in T_e\mathcal{D}^{s}_{\omega}$ the Euler-Arnold equations take the form

\begin{equation*}
(1+\Delta)v_t+P_{e}(\n_v(1+\Delta)v +(\n v)^t(\Delta)v)=0
\end{equation*}
\end{proof}
%
%
%
%

%
%
%
%
\begin{The}
If $s>n+3$, then the $H^1$ metric \eqref{eq:H1metric} and associated $H^1$ connection are smooth. Therefore, the Riemannian $H^1$ exponential map is smooth and gives a local diffeomorphism at the identity element of $\mathcal{D}^{s}_{\omega}$.
\end{The}
\begin{proof}
The $H^1$ metric \eqref{eq:H1metric} on the full diffeomorphism group $\mathcal{D}^s$ has been shown to be smooth in \cite{Shkoller98}. Since the $H^1$ metric on $\mathcal{D}^{s}_{\omega}$ is obtained by restriction of the $H^1$ metric on $\mathcal{D}^s$ we get a smooth metric on the symplectomorphism group.

For the existence of a smooth connection we rewrite the Euler-Arnold equation in Proposition 3.1 in terms of the flow $\eta$, and show that the resulting second order (geodesic) equation is a smooth ODE. In order to do this we first introduce the notation $L_{\eta}$ for any linear operator $L$ on functions or vector fields on $M$. Define $L_{\eta}u=L(u\circ\eta^{-1})\circ\eta$. Then, using the formulas $v=\dot{\eta}\circ\eta^{-1}$ and $\frac{D}{dt}\frac{d\eta}{dt}=(\partial_{t}v+\nabla_{v}v)\circ\eta$ we rewrite the Euler-Arnold equation of Proposition 3.1 as
\begin{align*}
\frac{D}{dt}\frac{d\eta}{dt}&=(1+\Delta)_{\eta}^{-1}P_{\eta}[(\nabla\dot{\eta}\circ{\eta^{-1}})^{\dagger}_{\eta}(1+\Delta)_{\eta}\dot{\eta}]\\
&-(1+\Delta)_{\eta}^{-1}P_{\eta}((\nabla_{\dot{\eta}\circ\eta^{-1}})_{\eta}(1+\Delta)_{\eta}\dot{\eta}-(1+\Delta)_{\eta}(\nabla_{\dot{\eta}\circ\eta^{-1}})_{\eta}\dot{\eta}).
\end{align*}
We begin with the first line of this expression. If $D$ is a first order differential operator on vector fields, then the conjugated operator $D_{\eta}$ is continuous as a map from vector fields of class $H^m$ to those of class $H^{m-1}$, as long as $\eta$ is a diffeomorphism of class $s>n+1$ and $0\leq m \leq s$; see Ebin and Marsden \cite{Ebin-Marsden} for details on this. In fact, if $\partial_{i}$ is any partial derivative, a direct computation shows that $(\partial_{i})_{\eta}f=(\partial_{i}\eta^{j})^{-1}\partial_{i}f$, where $^{-1}$ means matrix inverse. So $(\partial_{i})_{\eta}$ is smooth as a map from $H^s(M,R)\times\mathcal{D}^s\rightarrow H^{s-1}(M,R)$. Applying this to our situation we observe that the term in the square brackest is a first order differential operator, conjugated by $\eta$ and is hence smooth by the above reasoning. The operator $((1+\Delta)(\dot{\eta}\circ\eta^{-1})\circ\eta)$ is smooth as a map into vector fields of class $H^{s-2}$. Since $s-1>n+2$ and the term in the square bracket is $H^{s-1}$, the product of terms in $H^{s-1}$ and $H^{s-2}$ is again in $H^{s-2}$. Smoothness of the $L^2$ projection operator operator $P_{\eta}$ was established in \cite{Ebin-Marsden} and since all remaining terms involve multiplication, the application of $dR_{\eta}(I+\Delta)^{-1}dR_{\eta^{-1}}$ maps smoothly back into $H^s$ by general principles of bundle maps.
To analyze the second line observe that this is the commutator of a first order differential operator and a second order differential operator and thus is of order 2. In particular, calculating this commutator explicitly one sees that all terms involve multiplication by certain $\eta$ - conjugated derivatives of $\partial_{t}\eta$, which are bounded in the Sobolev topology under the assumption that $s>n+3$.
Having established smoothness of the right-hand side of the Lagrangian form of the geodesic equation the Theorem follows from the fundamental theorem of ODE's on Banach manifolds, cf. \cite{L}.
\end{proof}

\section{$H^{1}$ Geodesic Completeness of $\mathcal{D}^{s}_{\omega}$}

To prove global existence of $H^1$ geodesics we will use a form of the Euler-Arnold equations which is different to that given in Proposition 3.1. It is convenient to write elements of $T_e\mathcal{D}^s_{\omega}$ using the almost complex structure $J$. Indeed, any element of $T_e\mathcal{D}^s_{\omega}$ may be written as
\begin{equation*}
v=J\nabla F + h
\end{equation*}
for some $H^{s+1}$ function $F$ on $M$, with zero mean, and a harmonic vector field $h$.
\begin{Lem}\label{lemma1}
The $H^1$ algebra coadjoint operator on $T_{e}\mathcal{D}_{\omega}^{s}$
is given by
\begin{equation}
\mathrm{ad}_{v}^{*}w=(1+\Delta)^{-1}P\left((1+\Delta)\Delta g\cdot Jv\right),
\end{equation}
where $P$ denotes $L^{2}$ orthogonal projection onto $T_{e}\mathcal{D}_{\omega}^{s}$ and $w=J\nabla g + h$ for some function $g$ and harmonic vector field $h$.
\end{Lem}
\begin{proof}

Let $X\in T_{e}\mathcal{D}_{\omega}^{s}$. Then
\begin{align*}
\left(\mathrm{ad}_{v}^{*}w,X\right)_{H^{1}}&=\left(w,\mathrm{ad}_{v}X\right)_{H^{1}}\\
&=\int_{M}g\left(\left(1+\Delta\right)w,\mathrm{ad}_{v}X\right)\,d\mu\\
&=\int_{M}g\left(J\nabla(1+\Delta)g + h,J\nabla\omega(v,X)\right)\,d\mu\\
&=\int_{M}g\left(\nabla(1+\Delta)g,\nabla\omega(v,X)\right)\,d\mu\\
&=-\int_{M}\Delta(1+\Delta)g\cdot\omega(v,X)\,d\mu\\
&=-\int_{M}\omega(\left[(1+\Delta)\Delta g\right]v,X)\,d\mu\\
&=\int_{M}g\left(P\left[(1+\Delta)\Delta g\right]Jv,X\right)\,d\mu.
\end{align*}
Observe that the expression after the last equality is the $L^2$ metric and $X$ is in $T_{e}\mathcal{D}_{\omega}^{s}$. This is why we must compose with the projection $P_e$. This last expression can finally be written as
\begin{equation*}
=\left((1+\Delta)^{-1}P\left(\left[(1+\Delta)\Delta g\right]Jv\right),X\right)_{H^{1}}
\end{equation*}
which proves the Lemma.
\end{proof}

To prove Theorem 1.1 we shall show that the $H^s$ norm of the local solution $v$ remains bounded on any finite time interval. Using Lemma 4.1, we can estimate the growth of $v(t)$ in the $H^{s}$ norm by
\begin{align*}
\partial_{t}\frac{1}{2}\left(v,v\right)_{H^{s}}&=\left(-\mathrm{ad}_{v}^{*}v,\sum_{k=0}^{s}\Delta^{k}v\right)_{L^{2}}\\
&=\left(-(1+\triangle)^{-1}P\left((1+\triangle)\triangle f\cdot Jv\right),\sum_{k=0}^{s}\triangle^{k}v\right)_{L^{2}}\\
&\leq\left\Vert (1+\triangle)^{-1}P\left((1+\triangle)\triangle f\cdot Jv\right)\right\Vert _{L^{2}}\left\Vert v\right\Vert _{H^{s}}\\
&\lesssim\left\Vert (1+\triangle)\triangle f\cdot Jv\right\Vert _{L^{2}}\left\Vert v\right\Vert _{H^{s}}\\
&\lesssim\left\Vert f\right\Vert _{C^{4}}\cdot\left(v,v\right)_{H^{s}}
\end{align*}
so that if the $C^{4}$ norm of the potential $f$ remains bounded
on any finite time interval then so does the $H^{s}$ norm of $v$.

Multiplying both sides of the Euler-Arnold equation \eqref{eq:E-A} by $\mathrm{Ad}_{\eta(t)}^{*}$, and using \eqref{eq:gacoadjoint} we obtain the conservation law
\begin{equation} \label{eq:conservation}
\mathrm{Ad}_{\eta(t)}^{*}v(t)=v_{o}.
\end{equation}
Therefore, the integral
\begin{align*}
\int_{M}\eta^{*}\left(g^{\flat}\left(\left(1+\triangle\right)v(t)\right)\right)(w)\,d\mu&=\left(v(t),\mathrm{Ad}_{\eta(t)}w\right)_{H^{1}}\\
&=\left(\mathrm{Ad}_{\eta(t)}^{*}v(t),w\right)_{H^{1}}\\
&=const.
\end{align*}
 is constant for all $w\in T_{e}\mathcal{D}_{\omega}^{s}$. It follows
from the Hodge decomposition \eqref{eq:Hodge} that there exists a 2-form $\lambda(t)$
such that
\begin{equation*}
g^{\flat}\left(\left(1+\triangle\right)v_{o}\right)=\eta^{*}g^{\flat}\left(\left(1+\triangle\right)v(t)\right)+g^{\flat}\omega^{\sharp}\left(\delta\lambda(t)\right).
\end{equation*}
Applying $\delta\omega^{\flat}g^{\sharp}$ to both sides gives
\begin{equation*}
\delta\omega^{\flat}\left(1+\triangle\right)v_{o}=\delta\omega^{\flat}g^{\sharp}\eta^{*}g^{\flat}\left(1+\triangle\right)v(t).
\end{equation*}
The left hand side reduces to $\delta\omega^{\flat}\left(1+\triangle\right)v_{o}=\delta\left(d(1+\triangle)f_{o}\right)=\triangle(1+\triangle)f_{o}$
while the right-hand side is computed as $\left(\triangle(1+\triangle)f(t)\right)\circ\eta(t)$
using that $\delta\alpha=\ast di_{g^{\sharp}\alpha}\omega^{n}$ for
any one-form $\alpha$ (cf. \cite{Ebin} for details of this computation). We conclude that
\[
\triangle(1+\triangle)f_{o}=\left(\triangle(1+\triangle)f(t)\right)\circ\eta(t).
\]

Since $v_{o}\in H^{s}$ for $s>n+3$ there exists a positive $\alpha$
such that $v_{o}\in C^{3,\alpha}$ and hence $\triangle f(0)\in C^{2,\alpha}$.
Therefore
\begin{equation}
\triangle(1+\triangle)f(0)\in C^{0,\alpha}\label{eq:continuousholder}
\end{equation}
and it follows that $\left\Vert \triangle(1+\triangle)f(t)\right\Vert _{C^{0}}$
is bounded independently of $t$. In particular, $\left\Vert \triangle f(t)\right\Vert _{C^{0}}$
is bounded independently of $t$.
\begin{Lem} \label{lemma2}
If $\triangle g=q$ is a continuous function on $M$ with $\left\Vert q\right\Vert _{C^{0}}\leq Q$
then $dg$ is quasi-Lipschitz; that is, in each chart $U$, $dg$
satisfies the inequality
\begin{equation}
\left|dg(x)-dg(y)\right|\leq C\left|x-y\right|\log\left(\frac{1}{\left|x-y\right|}\right)
\end{equation}
with $\left|x-y\right|\leq1$.
\end{Lem}
\begin{proof}
Kato \cite{Kato} Lemma 1.2 or Morrey \cite{Morrey} Theorem 2.5.1
\end{proof}
From Lemma \ref{lemma2} it must be that $df(t)$ is quasi-Lipschitz uniformly
in $t$ aswell.
\begin{Lem} \label{lemma3}
Suppose $v$ is uniformly quasi-Lipschitz on $\left(-\epsilon,\epsilon\right)$
with flow $\eta(t)$. Then, working in charts, we can find positive
constants $C$ and $\beta$ such that
\begin{equation}
\left|\eta(t)(x)-\eta(t)(y)\right|\leq C\left|x-y\right|^{\beta},
\end{equation}
for $x$ and $y$ in the same chart and for all $t\in\left(-\epsilon,\epsilon\right)$.
\end{Lem}
\begin{proof}
Kato \cite{Kato}, section 2.5
\end{proof}
By Lemma \ref{lemma3} and \eqref{eq:continuousholder} it follows that $\triangle(1+\triangle)f(t)\in C^{0,\alpha'}$,
for a possibly smaller $\alpha'$, while standard elliptic estimates
give $f(t)\in C^{4,\alpha'}$. This completes the proof of Theorem 1.1.

%
%
%
%
%
\section{Fredholm Properties of the $H^1$ Exponential Map}
To begin this section we collect a few well known facts about Fredholm mappings.

Let $\eta$ be the $H^1$ geodesic in $\mathcal{D}^{s}_{\omega}$ emanating from the identity in the direction $v_{o} \in T_{e} \mathcal{D}^{s}_{\omega}$. In order to study conjugate points along $\eta$ it is convenient to express the derivative of the $H^1$ exponential map, at $tv_o$, in terms of solutions to the Jacobi equation

$$\bar{\nabla}^{1}_{\dot{\eta}(t)}\bar{\nabla}^{1}_{\dot{\eta}(t)}J(t) + \bar{R}^{1}(J(t),\dot{\eta}(t))\dot{\eta}(t)=0$$

along $\eta(t)=\exp_e(tv_o)$, with initial conditions $J(0)=0$, $J'(0)=w_o$. Here, $\bar{\nabla}^{1}$ is the right-invariant $H^1$ Levi-Civita connection on $\mathcal{D}^{s}_{\omega}$  and $\bar{R}^1$ the right-invariant $H^1$ curvature tensor on $\mathcal{D}^{s}_{\omega}$ given by

$$\bar{R}^{1}(X,Y,Z) = \bar{\nabla}^{1}_{X}\bar{\nabla}^{1}_{Y} Z - \bar{\nabla}^{1}_{Y}\bar{\nabla}^{1}_{X} Z - \bar{\nabla}^{1}_{[X,Y]} Z$$

for $X=u\circ \eta$, $Y=v\circ \eta$, $Z=w\circ \eta$, with $u$, $v$, $w \in T_{e}\mathcal{D}^{s}_{\omega}$.

\begin{Pro}
Let $M$ be a compact, $2n$ dimensional, Riemannian manifold with compatible symplectic form $\omega$. Then, for $s>n+3$ and $\eta \in \mathcal{D}^{s}_{\omega}$ the curvature tensor $\bar{R}^{1}_{\eta}$ is a bounded multilinear operator in the $H^s$ topology.
\end{Pro}

\begin{proof}
Shkoller \cite{Shkoller98} gave an explicit formula for the $H^1$ covariant derivative $\nabla^{1}$ on the full diffeomorphism group $\mathcal{D}^{s}$, see Theorem 3.1. The $H^1$ covariant derivative on $\mathcal{D}^{s}_{\omega}$ is then given by $\bar{\nabla}^{1}=P\circ\nabla^{1}$. To prove boundedness of the curvature tensor $\bar{R}^{1}$ one first proves boundedness of the curvature tensor $R^1$ on the full diffeomorphism group. The computations are tedious and do not differ in this situation so we refer the reader to \cite{Shkoller98} Lemma 4.1 through Proposition 4.1. To prove boundedness of the curvature tensor $\bar{R}^{1}$ we use the Gauss-Codazzi equations which relates the curvature of the submanifold to that of the ambient manifold. To estimate the additional terms in the Gauss-Codazzi equations one proceeds just as in Theorem 4.1 \cite{Shkoller98}, with the minor modification of applying the projection $P_{\eta}$ onto $T_{\eta}\mathcal{D}^{s}_{\omega}$ and $Q_{\eta}=I-P_{\eta}$, which both depend smoothly on the basepoint $\eta$ (cf. \cite{Ebin}) and define pseudo-differential operators of class $OPS^{0}_{1,0}$.
\end{proof}

As a consequence, Jacobi fields exist, are unique, and globally defined along $H^1$ geodesics in $\mathcal{D}^{s}_{\omega}$.

Define the Jacobi field solution operator $\Phi_t$ by
\begin{equation} \label{eq:dexp}
w_0 \to \Phi_t w_0
=
D\exp_{e}(tv_0)tw_0 = J(t),
\end{equation}
where $\exp_{e}$ is the globally defined Riemannian $H^1$ exponential map and $D\exp_{e}$ it's derivative.

\hskip 0.5cm

\begin{Pro} \label{prop:OLD}
Let $v_0 \in T_e\mathcal{D}_{\omega}^{s}$ and let $\eta(t)$ be the geodesic of the $H^1$ metric \eqref{eq:H1metric}
in $\mathcal{D}_{\omega}^{s}$ starting from the identity $e$ with velocity $v_0$.
Then $\Phi_t$ defined in \eqref{eq:dexp} is a family of bounded linear operators from
$T_{e}\mathcal{D}_{\omega}^{s}$ to $T_{\eta(t)}\mathcal{D}_{\omega}^{s}$.
Furthermore, if $v_0 \in T_e\mathcal{D}_{\omega}^{s+1}$ then $\Phi_t$ can be represented as
\begin{equation} \label{eq:solop}
\Phi_t = D\eta(t)\big( \Omega_t-\Gamma_t \big)
\end{equation}
where $\Omega_t$ and $\Gamma_t$ are bounded operators on $T_e\mathcal{D}_{\omega}^s$
given by
\begin{align}
&\Omega_t
=
\int_{0}^{t}\mathrm{Ad}_{\eta(\tau)^{-1}}\mathrm{Ad}_{\eta(\tau)^{-1}}^{*}\, d\tau
\label{eq:Omega} \\  \label{eq:Gamma}
&\Gamma_t
=
\int_{0}^{t}\mathrm{Ad}_{\eta(\tau)^{-1}}K_{v(\tau)}dR_{\eta^{-1}(\tau)}\Phi_\tau \, d\tau
\end{align}
and $K_v$ is a compact operator on $T_e\mathcal{D}_{\omega}^s$ given by
\begin{align} \label{eq:K}
w \to K_{v(t)} w = \mathrm{ad}_w^\ast v(t),
\quad
w \in T_e\mathcal{D}_\mu^s
\end{align}
and where $v(t)$ solves the Lagrangian averaged symplectic Euler equations \eqref{eq:ASE}.
\end{Pro}
\begin{proof}
The proof of the decomposition \eqref{eq:solop} may be found in \cite{MP}, Theorem 5.6.
We prove compactness of the operator $K_v$ for $v\in T_{e}\mathcal{D}^{s+1}_{\omega}$. From Lemma 4.2 we know that
$$K_{v}w = (1+\Delta)^{-1}P((1+\Delta)\Delta f Jw)$$
where $v=J\nabla f + h$ for some $H^{s+2}$ function $f$ and harmonic vector field $h$. We can approximate $v$ in the $H^{s+1}$ norm by a sequence of smooth vector fields $v_{k}$ such that $K_{v_k} \rightarrow K_v$ in the $H^s$ operator norm, cf. Lemma 3.4 \cite{Benn4}. Since a limit of compact operators is compact it suffices to prove that $K_v$ is compact when $v$ is smooth.
Observe that, since $(1+\Delta)\Delta f$ is smooth, $P((1+\Delta)\Delta f Jw)$ is continuous in the $H^s$ topology on $T_{e}\mathcal{D}^{s+1}_{\omega}$. Since $(1+\Delta)^{-1}$ is compact by the Rellich embedding Lemma, and the set of compact operators form a closed two-sided ideal, the composition $K_{v}$ is also a compact operator.
\end{proof}
\begin{Rem} \upshape
Note that the decomposition \eqref{eq:solop}-\eqref{eq:K} must be applied with care.
This is due to the loss of derivatives involved in calculating the differential of the left translation operator
$\xi \to L_\eta \xi$ and consequently of the adjoint operator in \eqref{eq:Ad}. This is why we consider $v_0$, and hence $\eta(t)$, in $H^{s+1}$ rather than $H^s$.
\end{Rem}
\begin{Rem}
To prove compactness of $K_{v}$ one may also proceed just as in Lemma 3.5 in \cite{Benn4}.
\end{Rem}

\begin{The} \label{prop:OLDFred}
Suppose $M$ is any closed, symplectic manifold of dimensions $2n$ with compatible Riemannian metric $g$. For any $v_0 \in T_e\mathcal{D}_{\omega}^s$, with $s>n+3$, the derivative $d\exp_e(tv_0)$ extends to
a Fredholm operator on the $H^1$-completions
$\overline{T_e\mathcal{D}_{\omega}}^{_{H^1}}$
and~
$\overline{T_{\exp_e(tv_0)}\mathcal{D}_\omega}^{_{H^1}}$.
\end{The}
\begin{proof}
To prove invertibility of \eqref{eq:Omega} we first note that since $\eta$ is a diffeomorphism it follows that $\mathrm{Ad}_{\eta}$
is invertible on $T_{e}\mathcal{D}_{\omega}^{1}$ with inverse $\mathrm{Ad}_{\eta}^{-1}=\mathrm{Ad}_{\eta^{-1}}$. Consequently, its $H^1$ adjoint is also invertible and satisfies $(\mathrm{Ad}^{*}_{\eta})^{-1}=\mathrm{Ad}^{*}_{\eta^{-1}}.$

Now, for any $w\in T_{e}\mathcal{D}_{\omega}^{1}$ we have
\begin{align*}
\left\langle w,\Omega(t)w\right\rangle _{H^{1}}&=\int_{0}^{t}\left\langle w,\mathrm{Ad}_{\eta^{-1}(s)}\mathrm{Ad}_{\eta^{-1}(s)}^{*}\right\rangle _{H^{1}}\,d\tau\\
&=\int_{0}^{t}\left\langle \mathrm{Ad}_{\eta^{-1}(s)}^{*}w,\mathrm{Ad}_{\eta^{-1}(s)}^{*}\right\rangle _{H^{1}}\,d\tau\\
&\geq\left(\int_{0}^{t}\frac{ds}{\left\Vert \mathrm{Ad}_{\eta(s)}^{*}\right\Vert _{H^{1}}}\right)\left\Vert w\right\Vert _{H^{1}}^{2}
\end{align*}
which implies that $\Omega(t)$ is bounded from below (and self-adjoint)
and hence invertible on $T_{e}\mathcal{D}_{\omega}^{1}$.
Compactness of the operator \eqref{eq:Gamma} on $\overline{T_e\mathcal{D}_{\omega}}^{_{H^1}}$ follows from compactness of the operator $K_{v}$, and hence compactness of the composition appearing under the integral in \eqref{eq:Gamma}, and finally from viewing the integral as a limit of sums of compact operators. This represents $D\exp_e(tv_0)$ as the sum of an invertible operator and compact operator which implies $D\exp_e(tv_0)$ is Fredholm of index zero.
\end{proof}
\hskip 0.5cm
We now sketch the main ideas in the proof of Fredholmness in the strong $H^s$ topology. We first assume that $\eta$, along with its initial velocity $v_o$, are $C^{\infty}$. That \eqref{eq:Gamma} is compact in the strong $H^s$ topology follows the same argument above. To obtain invertibility on $T_{e}\mathcal{D}_{\omega}^{s}$ one proves an estimate of the form
\begin{equation}
\left\Vert w\right\Vert _{H^{s}}\leq c_{1}\left\Vert \Omega(t)w\right\Vert _{H^{s}}+c_{2}\left\Vert w\right\Vert _{H^{s-1}},\label{eq:invertibilityestimate}
\end{equation}
which implies $\Omega(t)$ has closed range
and finite dimensional kernel on $T_{e}\mathcal{D}_{\omega}^{s}$.
But since $T_{e}\mathcal{D}_{\omega}^{s}$ is a subset of $T_{e}\mathcal{D}_{\omega}^{1}$
it follows that $\Omega(t)$ has empty kernel. To prove that $\Omega(t)$
has empty cokernel observe that for $u\in T_{e}\mathcal{D}_{\omega}^{s}$
there exists $w\in T_{e}\mathcal{D}_{\omega}^{1}$ such that $\Omega(t)w=u$
and the estimate (\ref{eq:invertibilityestimate}) inductively implies
that the $H^{s}$ norms of $w$ remain bounded aswell. That is, $w\in T_{e}\mathcal{D}_{\omega}^{s}$. The proof of this estimate can be obtained as in Theorem 7.1 of \cite{MP}, using the fact that $(I+\Delta)$, $(I+\Delta)^{-1}$, and $P_e$ are pseudo-differential operators of order 2, -2, and 0, respectively, and exploiting the appearance of commutators involving pseudo-differential operators. Alternatively, one may obtain this estimate following Lemma 3.2 and 3.3 of \cite{Benn4}. This then proves that $\Phi_{t}$, and hence $D\exp_{e}(tv_o)$, can be expressed as the sum of an invertible and compact operator, and is therefore Fredholm of index zero along $C^\infty$ geodesics.

To obtain Fredholmness along $H^s$ geodesics we approximate the $H^s$ initial velocity $v_o$ (and hence the $H^s$ geodesic $\eta(t)=\exp_{e}(tv_o)$) by a smooth vector field $\tilde{v}_o$ (respectively, smooth geodesic $\tilde{\eta}$). Since the exponential map is smooth in the $H^s$ topology, its derivative depends smoothly on $v_o$. In particular, it is locally Lipschitz and satisfies $\|D\mathrm{exp}_{e}(tv) - D\mathrm{exp}_{e}(t\tilde{v})\|_{L(H^s)} \leq C \|v-\tilde{v}\|_{H^s}$ uniformly on any time interval. Using \eqref{eq:invertibilityestimate} together with the decomposition \eqref{eq:solop} of the solution operator $\Phi_{t}$, we obtain the estimate
\begin{equation*}
\|w\|_{H^s} \leq \|\Phi_{t}w\|_{H^s}+\|w\|_{H^{s-1}} + \|\Gamma_{t}w\|_{H^s}.
\end{equation*}
Letting $\tilde{\Phi}_t$ and $\tilde{\Gamma}_t$ be the operators corresponding to $\tilde{v}_o$ and $\tilde{\eta}(t)$ above we obtain, for any $w \in T_{e}\mathcal{D}^{s}_{\omega}$,
\begin{align*}
\|w\|_{H^s} &
\leq \|\tilde{\Phi}_{t}w\|_{H^s} + \|w\|_{H^{s-1}} + \|\tilde{\Gamma}_{t}w\|_{H^s}\\
&\leq \|D\mathrm{exp}_{e}(tv)w\|_{H^s} + \|D\mathrm{exp}_{e}(t\tilde{v})w-D\mathrm{exp}_{e}(tv)\|_{H^s} + \|w\|_{H^{s-1}} + \|\tilde{\Gamma}_{t}w\|_{H^s}\\
&\leq \|D\mathrm{exp}_{e}(tv)w\|_{H^s}+\epsilon \|w\|_{H^s}+ \|w\|_{H^{s-1}} + \|\tilde{\Gamma}_{t}w\|_{H^s}.
\end{align*}
Choosing $\epsilon$ small enough we conclude that $D\mathrm{exp}_{e}(tv)$ has closed range and finite dimensional kernel. Thus $D\mathrm{exp}_{e}(tv)$ is semi-Fredholm and we can define its index, cf. Kato \cite{Kato66}. Since $D\mathrm{exp}_{e}(tv)$ is the identity at $t=0$, and the index is continuous on the space of Fredholm maps, $D\mathrm{exp}_{e}(tv)$ must have the same index for all $t$. Thus $D\mathrm{exp}_{e}(tv)$ is Fredholm of index zero. In particular,

%
%
%
%
\section{Stationary solutions and conjugate points on $(\D^s_\omega(\mathbb{CP}^n),\<,\>_1)$}
In this section we show that any killing vector field on $\mathbb{CP}^n$ generates a stationary solution of the averaged symplectic Euler equation \eqref{eq:ASE} and such a geodesic contains conjugate points. However we remark that generally, killing vector fields on arbitrary  symplectic manifolds do not generate solutions for the averaged symplectic Euler equation Euler equation.
The isometry group of a symplectic manifold $(M,\omega)$ with compatible Riemannian metric $g$, denoted by $\mathrm{Iso}(M)$, consists of those diffeomorphisms which satisfy
\begin{equation*}
\eta^{*}\omega=\omega \quad \eta^{*}g=g.
\end{equation*}
The isometry group is contained in the group of symplectomorphisms as a finite dimensional Lie subgroup with Lie algebra $T_e\mathrm{Iso}(M)={v\in T_e\mathcal{D}^{s}_{\omega}:\mathcal{L}_{v}g=0}$. Elements of $T_e\mathrm{Iso}(M)$ are called Killing vector fields and satisfy
\begin{equation} \label{eq:kill}
g(\nabla_{u}v,w)=-g(u,\nabla_{w}v).
\end{equation}

Suppose $(M,\omega)$ is also an Einstein manifold; that is there exists a constant $\lambda$ such that $Ric(X,Y)=\lambda g(X,Y)$. Then
\begin{equation}
\L_rv=\textbf{Ric}v
\end{equation}
where $\L_r=-\textrm{tr}\nabla^2$ is the rough Laplacian, which is related to the Hodge Laplacian by $\L=\L_r+\textbf{Ric}$. Consequently $\L v=\L_r v+\textbf{Ric}v=2\textbf{Ric}v$ for any killing vector field and
\begin{equation*}
g(\textbf{Ric}v,X)=Ric(v,X)=\lambda g(v,X)=g(\lambda v,X).
\end{equation*}
It follows that $\textbf{Ric}v=\lambda v$ and
\begin{equation} \label{eq:Leigenfn}
\Delta v=2\lambda v
\end{equation}

\begin{Pro}\label{Pro Killing fields on Einstein Kahler manifolds}
Let $(M,\omega,g)$ be a symplectic Einstein manifoldand let $v\in T_e\D^s_\omega(M)$ be a killing vector field on $M$. Then $v$ defines a time-independent solution of the averaged symplectic Euler equation \eqref{eq:ASE} and the corresponding geodesic $\eta(t)$ consists of isometries for each $t$.
\end{Pro}
\begin{proof}
Let $v$ be a Killing vector field and $w$ any vector field on $M$. Then, using \eqref{eq:kill} and \eqref{eq:Leigenfn}, we compute
\begin{align*}
\langle \nabla(I+\Delta) v + (\nabla v)^{\dagger}\Delta v, w\rangle_{0} &= \langle \nabla(1+2\lambda) v + (\nabla v)^{\dagger}2\lambda v, w\rangle_{0}\\
&=\langle \nabla_{v}v,w\rangle_{0}\\
&=\langle -\frac{1}{2}\mathrm{grad} |v|^2,w\rangle_{0}
\end{align*}
Since this holds for any $w$ we have $P_{e}(\nabla(I+\Delta) v + (\nabla v)^{\dagger}\Delta v)=0$ and $v$ is a time-independent solution to the averaged symplectic Euler equations.
If $\eta(t)$ is the corresponding geodesic in $\mathcal{D}^{s}_{\omega}$ then

\begin{equation}
\frac{d}{dt}\eta(t)^{*}g=\eta(t)^{*}\mathcal{L}_{v}g=0
\end{equation}

so that $\eta(t)$ is an isometry for each $t$.
\end{proof}
%

\begin{The}
The $H^1$ metric on $\mathcal{D}^{s}_{\omega}(\mathbb{CP}^{n})$ carries conjugate points for all $n\geq 2$.
\end{The}
\begin{proof}
The manifold $M=\mathbb{CP}^n$ is an Einstein K\"{a}hler manifold with the Fubini-Study metric, which is given in components as
\begin{equation*}
h_{i\bar{j}}=h(\partial_i,\bar{\partial}_j)=\frac{  (1+|z|^2)\delta_{i\bar{j}} -\bar{z}_i z_j  }{(1+|z|^2)}
\end{equation*}
where $z=(z_1,\dots,z_n$ is a point in $\mathbb{CP}^n$ and $|z|^2=|z_1|^2+\dots+|z_n|^2$. The isometry group of $\mathbb{CP}^{n}$ is $PU(n+1)$, the projective unitary group, which is the quotient of the unitary group by its center, embedded as scalars. In terms of matrices, the unitary group consists of $(n+1)\times(n+1)$ complex matrices whose center consists of elements of the form $e^{i\theta}I$ so that $\mathrm{Iso}(\mathbb{CP}^n)$ consists of equivalence classes of matrices where two matrices $A$ and $B$ are equivalent if $A=e^{i\theta}I\times B$.

If $n$ is even, consider the following 2-parameter variation of isometries
\begin{equation*}
\gamma(s,t)=A(s)B(t)A(s)^{-1},
\end{equation*}

where
$$A(s)=\left(\begin{array}{ccccc} i &  &   &  & \\
 & \mathbb{A}(s) &   &  0& \\
 &  & \mathbb{A}(s)  &  &\\
 & 0 &   & \ddots &\\
 &  &   &  & \mathbb{A}(s)\\
\end{array}\right),
B(t)=\left(\begin{array}{ccccc} \mathbb{B}(t) &  &   &  & \\
 & \mathbb{B}(t) &   &  0& \\
 &  & \ddots  &  &\\
 & 0 &   & \mathbb{B}(t) &\\
 &  &   &  & i\\
\end{array}\right)$$
and  $\mathbb{A}(s)=\left(\begin{array}{cc}  i\textrm{cos s}  & \textrm{sin s}\\ \textrm{sin s} & i\textrm{cos s}\end{array}\right)$
and $\mathbb{B}(t)=\left(\begin{array}{cc}  i\textrm{cos t}  & \textrm{sin t}\\ \textrm{sin t} & i\textrm{cos t}\end{array}\right)$ are $2\times 2$ block matrices. Observe that $\gamma(s,0)=iI=\equiv I$ for all $s$.
With composition as matrix multiplication and inversion as matrix inversion, a direct calculation shows that the matrix associated with the vector field $v(s,t)=\dot{\gamma}(s,t)\circ\gamma^{-1}(s,t)$ is skew-hermitian and therefore lies in the Lie algebra to $PU(n+1)$. Thus $v$ is a Killing vector field and defines a time-independent solution to the averaged symplectic Euler equations, by Proposition 6.1. Hence $\gamma(s,t)$ is a geodesic for each $s$. The variation field is given by
\begin{align*}
 J(t)&=\frac{d}{ds}\gamma(s,t)|_{s=0}\\
&=\left(\begin{array}{ccccccc} \mathbf{0}  & \mathbb{D}_1(t)  &   &  & & & \\
 -\mathbb{D}_1(t)& \mathbf{0} & \mathbb{D}_1(t)  &  &  & & \\
 & -\mathbb{D}_1(t) & \mathbf{0}  & \ddots  & & & \\
 &  &\ddots   & \ddots  &  \mathbb{D}_1(t) & &\\
&  &   & -\mathbb{D}_1(t)  & \mathbf{0} & \ddots  & \\
& &  &   & \ddots  & \ddots & \mathbb{D}_2(t)\\
& &  &   &  & -\mathbb{D}_2(t) & \mathbf{0}\\
\end{array}\right)
\end{align*}
where $\mathbf{0}=\left(\begin{array}{cc}  0  & 0 \\ 0 & 0\end{array}\right)$,
$\mathbb{D}_1(t)=\left(\begin{array}{cc}  -\textrm{sin t}  & 0\\ 0 & \textrm{sin t}\end{array}\right)$, and $\mathbb{D}_2(t)=\left(\begin{array}{cc}  -\textrm{sin t}  & 0\\ 0 & i(1-\textrm{cos t})\end{array}\right)$ are $2\times 2$ block matrices.
Since $J(0)=J(2\pi)=0$, $\gamma(0)$ is conjugate to $\gamma(2\pi)$.

The proof for odd $n$ uses exactly the same 2-parameter variation of curves above except we take

$$A(s)=\left(\begin{array}{cccccc} i &  &   &  & &\\
 & \mathbb{A}(s) &   &  0& &\\
 &  & \mathbb{A}(s)  &  &  & \\
 & 0 &   & \ddots &  &\\
 &  &   &  & \mathbb{A}(s)&\\
 &  &   &  &  &  i\\
\end{array}\right),
B(t)=\left(\begin{array}{cccccc} \mathbb{B}(t) &  &   &  & \\
 & \mathbb{B}(t) &   &  0& \\
 &  & \ddots  &  &\\
 & 0 &   & \mathbb{B}(t) &\\
 &  &   &  & iI_{2}\\
\end{array}\right) .$$
\end{proof}
%

%

\end{document}